\documentclass[11pt]{amsart}

\usepackage[left=1in,right=1in,top=1in,bottom=1in]{geometry}

\usepackage[all,cmtip]{xy}
\usepackage{amssymb}
\usepackage[english]{babel}
\usepackage{mathrsfs}

\usepackage{rotating}

\newtheorem{theorem}{Theorem}[section]
\newtheorem{lemma}[theorem]{Lemma}
\newtheorem{pro}[theorem]{Proposition}
\newtheorem{cor}[theorem]{Corollary}

\theoremstyle{definition}
\newtheorem{definition}[theorem]{Definition}

\newtheorem{example}[theorem]{Example}

\newtheorem{remark}[theorem]{Remark}

\newtheorem{question}[theorem]{Question}

\newtheorem{notation}[theorem]{Notation}

\newtheorem{claim}{Claim}
\def\ssp{selectively sequentially pseudocompact}
\def\sp{selectively pseudocompact}
\def\seqc{sequentially compact}

\def\cc{countably compact}
\def\sel{selectively $\mathcal{S}$}

\def\N{\mathbb{N}}

\def\OP{\mathsf{OP}}
\def\A{\mathsf{O}}
\def\B{\mathsf{P}}
\def\Sp{\mathsf{Sp}}
\def\Ssp{\mathsf{Ssp}}
\def\SIGMA{\mathsf{o}}
\def\TAU{\mathsf{p}}

\def\Seq{\mathrm{Seq}}
\def\disc{\mathrm{disc}}

\def\dom{\mathrm{dom}}

\begin{document}
\title
{Compactness properties defined by open-point games}

\author{A. Dorantes-Aldama}
\address{Departamento de Matem\'{a}ticas, Facultad de Ciencias, Circuito exterior s/n, Ciudad Universitaria, CP. 04510, CDMX, M\'{e}xico}
\email{alejandro\_dorantes@ciencias.unam.mx}
\thanks{The first listed author was supported by CONACyT, M\'exico: Estancia Posdoctoral al Extranjero, 178425/277660}

\author{D. Shakhmatov}
\address{Division of Mathematics, Physics and Earth Sciences\\
Graduate School of Science and Engineering\\
Ehime University, Matsuyama 790-8577, Japan}
\email{dmitri.shakhmatov@ehime-u.ac.jp}
\thanks{The second listed author was partially supported by the Grant-in-Aid for Scientific Research~(C) No.~26400091 by the Japan Society for the Promotion of Science (JSPS)}

\begin{abstract} 
Let $\mathcal{S}$ be a topological property of sequences (such as, for example, ``to contain a convergent subsequence'' or ``to have an accumulation point'').
We introduce the following open-point game  $\OP(X,\mathcal{S})$ on a topological space $X$. 
In the $n$th move, Player $\A$ chooses a non-empty open set   $U_n\subseteq X$, and Player $\B$ responds by 
selecting a point $x_n\in U_n.$ 
Player $\B$ wins the game if the sequence $\{x_n:n\in \N\}$ 
satisfies property $\mathcal{S}$ in $X$; otherwise, Player $\A$ wins.
The (non-)existence of regular or stationary winning strategies in  $\OP(X,\mathcal{S})$ for both players defines new compactness properties of the underlying space $X$. We thoroughly investigate 
these properties and construct examples distinguishing
half of them, for an arbitrary property $\mathcal{S}$ sandwiched between sequential compactness and countable compactness. 
\end{abstract}

\dedicatory{Dedicated to Professor Alexander V. Arhangel'ski\u{\i} \ on the occasion of his 80th anniversary}

\maketitle

{\em All topological spaces are assumed to be Tychonoff.\/}

The symbol $\N$ denotes the set of {\em positive\/} natural numbers.

A point $x$ is said to be an {\em accumulation point\/} of a sequence $\{x_n:n\in\N\}$ of points of a topological space $X$ provided that the set $\{n\in\N: x_n\in U\}$ is infinite for every open neighbourhood $U$ of $x$ in $X$.

\section{Introduction}

\begin{definition}\label{definition_1} 
A topological space $X$ is called:
\begin{itemize}
\item[(i)] {\em sequentially compact} if every sequence in $X$ has a convergent subsequence;
\item[(ii)] {\em countably compact} if every sequence in $X$ has an accumulation point in $X$;
\item[(iii)] {\em pseudocompact\/} if every real-valued continuous function defined on $X$ is bounded;
\item[(iv)] {\em \ssp} if for every sequence $\{U_n:n\in \N\}$ of non-empty open subsets of $X$,
we can choose a point $x_n\in U_n$ for every $n\in\N$ in such a way that the sequence $\{x_n:n\in\N\} $
has a convergent subsequence;
\item[(v)] {\em \sp} if and only if for every sequence $\{U_n:n\in \N\}$ of 
non-empty open subsets of $X$,
we can choose a point $x_n\in U_n$ for every $n\in\N$ in such a way that the sequence $\{x_n:n\in\N\} $ has an accumulation point in $X$.
\end{itemize}
\end{definition}

The properties (i)--(iii) are well known \cite{En}, while selective properties (iv) and (v) were introduced recently in \cite{DoS}. 
It was proved in \cite[Theorem 2.1]{DoS} that 
the property from item (v) is equivalent to 
the notion of strong pseudocompactness 
introduced earlier in~\cite{GO}. 

The following diagram summarizes relations between the properties 
from the above definition.
{\small
\begin{center} \medskip\hspace{1em}
\xymatrix{
\text{sequentially compact}\ar[r]\ar[d]_1 & \text{countably compact}\ar[d]_2 & \\
\text{selectively sequentially pseudocompact}\ar[r]
& \text{selectively pseudocompact}\ar[r]&  \text{pseudocompact}
}
\end{center}
\vspace{10pt}
\begin{center}
Diagram 1.
\end{center}
}

\medskip
None of the arrows in this diagram are reversible; see \cite{DoS}.

Selective sequential pseudocompactness is the only property in Diagram~1 which is fully productive; that is, any Tychonoff product of \ssp\ spaces is \ssp\ \cite{DoS}.
This productivity is a major advantage of selective sequential pseudocompactness when compared to other compactness-like properties.

In this paper, we define two open-point topological games closely related to the class of selectively (sequentially) pseudocompact spaces.
Let $X$ be a topological space.
At round $n$, Player $\A$ chooses a non-empty open subset $U_n$ of $X$, and Player $\B$ responds by selecting a point $x_n\in U_n$. 
In the {\em \ssp\ game\/} $\Ssp(X)$ on $X$, Player $\B$ wins if the sequence 
$\{x_n:n\in\N\}$ has a convergent subsequence; 
otherwise Player $\A$ wins. 
In the {\em selectively pseudocompact game\/} $\Sp(X)$ on $X$,
Player $\B$ wins if the sequence 
$\{x_n:n\in\N\}$ has an accumulation point in $X$;
otherwise Player $\A$ wins. 
The (non-)existence of winning strategies for each player
in the game $\Ssp(X)$ (in the game $\Sp(X)$) defines a compactness-like property of $X$ sandwiched between sequential compactness (countable compactness) and selective sequential pseudocompactness (selective pseudocompactness)
of~$X$.
In this way we develop a fine structure of the area represented by arrows~1 and~2  of Diagram~1. Furthermore, we construct examples showing that the newly introduced notions are mostly distinct. The most sophisticated example is  locally compact, first-countable, zero-dimensional space $X$ such that Player $\B$ has a winning strategy in $\Ssp(X)$ but 
does not have a stationary winning strategy even  in $\Sp(X).$
(A strategy for Player $\B$ is called {\em stationary\/} if it depends only on the last move $U_n$ of the opponent, and not on the whole 
sequence $(U_1,U_2,\dots,U_n)$.)
This example makes essential use of a van Douwen maximally almost disjoint family of subsets of $\N$  due to Raghavan
\cite[Theorem 2.14]{R}.

The paper is organized as follows.
In Section \ref{sec:2}, we introduce the general notion of a topological property of sequences $\mathcal{S}$ and give five examples of such properties in Example \ref{examples}.
Each topological property $\mathcal{S}$ of sequences gives rise to four natural properties of topological spaces; see Definition \ref{def:selectively:S} and Diagram~2. In Section \ref{sec:3} we show that arrows (a) and (b) in this diagram are not reversible (Example \ref{ex:3.3}).
In order to show that arrow (c) is not reversible either,
in Section \ref{sec:4}
we introduce a general open-point game 
$\OP(X,\mathcal{S})$ on $X$ similar to $\Ssp(X)$ and $\Sp(X)$ in which Player $\B$ is declared a winner 
when the sequence $\{x_n:n\in\N\}$ selected by $\B$ satisfies property $\mathcal{S}$ in $X$.
The arrow (c) from Diagram~2 decomposes into four different arrows
$(c_1)$--$(c_4)$
of Diagram~3.
In Section \ref{sec:5} we introduce games $\Ssp(X)$ and $\Sp(X)$ 
as particular cases of the general game $\OP(X,\mathcal{S})$.

Theorem \ref{from:winning:to:stationary} states that a winning strategy for Player $\A$ in $\OP(X,\mathcal{S})$ generates a stationary winning strategy for Player $\A$ in the game $\OP(Z,\mathcal{S})$ played on a product $Z$ of a space $X$ and the one-point compactification $Y$ of a discrete space of cardinality $|X|$.

In Section \ref{MAD:family},
we construct a maximal almost disjoint family $\mathscr{F}$ consisting of injections from an infinite subset of $\N$ to a fixed infinite set $D$ built in Theorem \ref{set-theoretic:result}.
When $D=\N$, the family $\mathscr{F}$ is nothing but an 
``injective version'' of a van Douwen MAD  family of D. Raghavan
\cite[Theorem 2.14]{R}.
The family $\mathscr{F}$ is used in the construction 
a locally compact first countable zero-dimensional space 
$X$ such that  Player $\B$ has a winning strategy in $\Ssp(X)$ but 
does not have a stationary winning strategy even  in $\Sp(X)$; see Theorem 
\ref{arrow:3}.
This example is consequently employed in 
Corollary \ref{cor:5.11} 
showing that arrow $(c_1)$ of Diagram 3
is not reversible. 
In Section \ref{sec:6}, we 
give an example showing that arrow $(c_4)$ of Diagram 3
is not reversible (Corollary \ref{not:arrow:c4}).
The reversibility of arrows $(c_2)$ and $(c_3)$ of Diagram 3 remains unclear; see Question
\ref{que:7.1}. The reversibility of arrow $(c_2)$  is equivalent to determinacy of the game
$\OP(X,\mathcal{S})$.

\section{Topological properties of sequences}
\label{sec:2}

For a set $X$,
we identify $X^\N$ with the set of all sequences $\{x_n:n\in\N\}$ in $X$.

\begin{definition}
{\em A topological property of sequences\/} is a class 
$\mathcal{S}=\{\mathcal{S}_X: X$ is a topological space$\}$,
where $\mathcal{S}_X\subseteq X^\N$
for every topological space $X$, 
such that
\begin{equation}
\label{eq:extensions}
\mathcal{S}_Y\subseteq \mathcal{S}_X
\text{ whenever }
Y
\text{ is a subspace of }
X.
\end{equation}

When $s\in \mathcal{S}_X$, 
we shall say that 
the sequence $s\in X^{\N}$ 
{\em satisfies property $\mathcal{S}$ in $X$\/}.
\end{definition}

\begin{remark}
\label{relative:remark}
Condition \eqref{eq:extensions} means that, for every topological property of sequences $\mathcal{S}$, if a sequence of points in a subspace $Y$ of a topological space $X$ satisfies $\mathcal{S}$ in $Y$, then it also satisfies $\mathcal{S}$ in $X$.
\end{remark}

\begin{definition}\label{def:selectively:S}
Let 
$\mathcal{S}$ be a 
topological property of sequences. We shall say that:
\begin{itemize}
\item[(i)] a topological space $X$ {\em satisfies $\mathcal{S} $\/} if $\mathcal{S}_X=X^\N$; that is,
if
every sequence 
in $X$ satisfies $\mathcal{S}$.

\item[(ii)] 
a subspace $Y$ of a topological space $X$ {\em relatively satisfies $\mathcal{S} $ in $X$\/} if 
$Y^\N\subseteq \mathcal{S}_X$;
that is, if
every sequence in $Y$ satisfies $\mathcal{S}$ in $X.$

\item[(iii)] a topological space $X$ is {\em \sel} 
if for every sequence 
$\{U_n:n\in\N\}$ of non-empty open subsets of $X,$
one can select a point $x_n\in U_n$ for $n\in \N$ 
in such a way that the resulting sequence $\{x_n:n\in\N\}$ satisfies $\mathcal{S}$
in $X$.
\end{itemize}
\end{definition}

Diagram~2 below
holds for every topological property of sequences $\mathcal{S}$.
The implication (a) is obvious, the implication (b) follows from 
Remark \ref{relative:remark}, and the implication (c) 
is straightforward.

\begin{center} 
\medskip\hspace{1em}
{\small
\xymatrix{
X \text{ satisfies } \mathcal{S}\ar[d]_{(a)}\\
X 
\text{ has a dense subspace satisfying } 
\mathcal{S}\ar[d]_{(b)}\\
X 
\text{ has a dense subspace that relatively satisfies } 
\mathcal{S}\ar[d]_{(c)}\text{ in }X\\
X \text{ is \sel }
}
}
\end{center}
\medskip
\begin{center}
Diagram 2.
\end{center}
\medskip

Let us mention five examples of topological properties of sequences $\mathcal{S}$.

\begin{example}
\label{examples}
For each topological space $X$, define
$\mathcal{S}_X$ to be the set of all sequences
$s\in X^\N$ such that:
\begin{itemize}
\item[(i)]
$s$ has
a subsequence $\{s(n_j):j\in\N\}$
converging to some point $x\in X$;
\item[(ii)]  
$s$ has
a subsequence $\{s(n_j):j\in\N\}$
whose closure in $X$ is compact;
\item[(iii)] 
for every free ultrafilter $p$ on $\N$,
there exists a subsequence $\{s(n_j):j\in\N\}$ of $s$ which 
$p$-converges to some point $x\in X$; that is,
 $\{j\in\N: s({n_j})\in V\}\in p$ for every open neighborhood $V$ of $x$ in $X$;
\item[(iv)] 
$s$ has
a subsequence $\{s(n_j):j\in\N\}$ which 
$p$-converges to some point $x\in X$, for a fixed ultrafilter $p$ on $\N$;
\item[(v)] 
$s$
has an accumulation point in $X$. 
\end{itemize}
\end{example}

\begin{remark}
\label{names:remark}
\begin{itemize}
\item[(i)]
For the property $\mathcal{S}$ defined in item (i) of 
Remark \ref{examples}, a topological space $X$ (relatively) satisfies 
$\mathcal{S}$ if and only if $X$ is (relatively) sequentially compact,
and the class of \sel\ spaces coincides with the class of selectively sequentially pseudocompact spaces.
\item[(ii)]
Let $\mathcal{S}$ be the property defined in item (ii) of 
Remark \ref{examples}. The topological spaces that satisfies $\mathcal{S}$ are called totally countably compact in
\cite{M}.
The topological spaces that are \sel\ are called totally pseudocompact in \cite[Definition 6]{M}.
\item[(iii)]
For the property $\mathcal{S}$ defined in item (v) of 
Remark \ref{examples}, a topological space $X$ (relatively) satisfies 
$\mathcal{S}$ if and only if $X$ is (relatively) countably compact,
and the class of \sel\ spaces coincides with the class of 
selectively pseudocompact spaces.
\end{itemize}
\end{remark}

\begin{definition}
\label{comparing:strength:of:R:and:S}
Given two topological properties of sequences $\mathcal{R}$ and
$\mathcal{S}$, we say that {\em $\mathcal{R}$ is stronger than
$\mathcal{S}$\/} (and {\em $\mathcal{S}$ is weaker than $\mathcal{R}$\/}) provided that $\mathcal{R}_X\subseteq \mathcal{S}_X$ for every topological space $X$; that is, if every sequence in a topological space
$X$ that satisfies $\mathcal{R}$ in $X$ also satisfies $\mathcal{S}$ in $X$.
\end{definition}

\begin{remark}
\label{decreasing:strength}
The properties in items (i)--(v) of Example \ref{examples}
are listed in the order of decreasing strength. Implication (ii)$\to $(iii) is due to Bernstein \cite{Be},
who proved that each sequence in a compact space $X$
has a  $p$-limit 
point for every free ultrafilter $p$ on $\N.$
\end{remark}

\begin{remark}\label{remark:pseudocompact}
Let $\mathcal{S}$ be a topological property of sequences which is stronger than item (v) of Example \ref{examples}.

(i) If $X$ satisfies $\mathcal S$, then $X$ is countably compact. 

(ii) If $X$ is \sel,
then $X$ is pseudocompact. 

\end{remark}

\section{On reversibility of arrows in Diagram 2}
\label{sec:3}

Let us recall two examples from the literature.

\begin{example}\label{plank}
Let $\alpha$ be an ordinal and $[0,\alpha)$ be the space of all ordinals less than $\alpha$ with the order topology. 
The space {\em $T= [0,\omega+1)\times [0,\omega_1+1)\setminus \{(\omega,\omega_1)\}$ (is first countable and)
has a dense \seqc\ subspace but is not countably compact\/}. 
Indeed, the sequence 
$\{(n,\omega_1):n\in \N\} $ does not have an accumulation point in $T$,
so $T$ is not countably compact.
Since 
$[0,\omega+1)$
is sequentially compact, 
$D=[0,\omega+1)\times [0,\omega_1)$ is sequentially compact.
Finally, note that $D$
is dense in $T.$ 
\end{example}

\begin{example}\label{mrowka}
Let $\mathcal{A}$ be an arbitrary maximal almost disjoint family
of subsets of $\N$. Consider
the Mr\'owka space $X=\N\cup \mathcal{A}$ associated with  $\mathcal{A}$ \cite[3.6.I]{En}.
Then {\em $X$ has a dense relatively \seqc\ subspace $\N$, yet 
does not contain any dense \cc\ subspace\/}. A straightforward verification of this fact is left to the reader.
\end{example}

We shall now present examples demonstrating that implications (a) and (b) of Diagram 2 are not reversible, for any topological property $\mathcal{S}$ of sequences whose strength lies in between items (i) and
(v) 
of Example \ref{examples}.

\begin{example}
\label{ex:3.3}
Let $\mathcal{S}$ be a topological property of sequences which is weaker than  item (i) but stronger than 
item (v) of Example \ref{examples}.
\begin{itemize}
\item[(i)]
The space $T$ from Example \ref{plank} has a dense \seqc\ subspace $Y$. Since $\mathcal{S}$ is weaker than sequential compactness, $Y$ satisfies $\mathcal{S}$. Since $\mathcal{S}$ is stronger than countable compactness and $T$ is not \cc, it follows that $T$ does not satisfy $\mathcal{S}$. This shows that {\em arrow (a) 
in Diagram 2 is not reversible\/}.

\item[(ii)] The space $X$ from Example \ref{mrowka} has a dense subspace $Y$ which is relatively \seqc\ in $X$. Since $\mathcal{S}$ is weaker than sequential compactness, $Y$
relatively satisfies $\mathcal{S}$ in $X$. Since $\mathcal{S}$ is stronger than countable compactness and $X$ does not have a dense \cc\ subspace,
it follows that no dense subspace of $X$ satisfies $\mathcal{S}$.
 This shows that {\em arrow (b) 
in Diagram 2 is not reversible\/}.
\end{itemize}
\end{example}

Our goal is 
to show that arrow (c) of Diagram 2 is not reversible either. 
In fact, we 
do even more. In the next section, we 
introduce a topological game $\OP(X,\mathcal{S})$ on a space $X$ such that:
\begin{itemize}
\item
  Player $\B$ 
has a stationary winning strategy in this game on $X$ if and only if 
$X$ contains a dense subspace $Y$ which relatively satisfies $\mathcal{S}$ in $X$, and
\item
If Player $\A$ does not have a stationary winning strategy in this game on $X$, then $X$ is \sel. 
\end{itemize}
In this way we 
develop a fine structure of the area represented by arrow (c) of Diagram~2; see Diagram~3.
Furthermore, we produce two examples showing that half of the newly introduced notions are distinct; see Corollaries \ref{cor:5.11} and \ref{not:arrow:c4} below.

\section{Open-point game $\OP(X,\mathcal{S})$}

\label{sec:4}

Let $\mathcal{S}$ be a a fixed topological property of sequences.
For a topological space $X$,
we define the 
open-point topological game $\OP(X,\mathcal{S})$ on $X$ between Player $\A$ and Player $\B$ as follows. 
An infinite sequence $w =(U_1,x_1,U_2,x_2,\ldots)$ such that
$U_n$ is a non-empty open set and $x_n\in U_n$ for every $n\in \N$ 
 is called 
a {\em play in\/} $\OP(X,\mathcal{S})$.
Given a play $w =(U_1,x_1,U_2,x_2,\ldots)$ in $\OP(X,\mathcal{S})$, we will say that {\em Player $\B$ wins $w $} if 
$\{x_{n}:n\in\N\}$ satisfies $\mathcal{S}$ in $X$, otherwise, {\em Player $\A$ wins $w.$} 

\begin{notation}
(i) Given a set $Y$, we use $\Seq(Y)$ to denote the set of all finite sequences $(y_1,\dots,y_n)$ of elements of $Y$.
We  include
the empty sequence $\emptyset$ in $\Seq(Y)$.

(ii)
For a topological space $X$, we use  $\mathcal{O}(X)$ to denote the family of all non-empty open subsets of $X$.
\end{notation}

\begin{definition}
\label{def:game:OP(X,S)}
(i) 
A function $\SIGMA :\Seq(X)\to\mathcal{O}(X)$ 
is called a
{\em  strategy for Player $\A$} in $\OP(X,\mathcal{S})$. 
 {\em A strategy for Player $\B$} in $\OP(X,\mathcal{S})$ is a function $\TAU: \Seq(\mathcal{O}(X))\setminus\{\emptyset\}\to X$ such that 
\begin{equation}
\label{eq:rule:of:the:game} 
\TAU(U_1,U_2,\dots,U_n)\in U_n
\text{ for every }
(U_1,U_2,\dots,U_n)\in \Seq(\mathcal{O}(X))\setminus\{\emptyset\}.
\end{equation}

(ii) 
Two strategies $\SIGMA $ and $\TAU $ for Players $\A$ and $\B$, respectively, {\em produce the play\/} 
$w_{\SIGMA,\TAU}$ in $\OP(X,\mathcal{S})$ as follows.
Player $\A$ starts with 
\begin{equation}
\label{eq:U_1}
U_1=\SIGMA(\emptyset),
\end{equation}
 and Player $\B$ responds with 
 \begin{equation}
\label{eq:x_1} 
 x_1=\TAU(U_1).
 \end{equation}
At the $n$th move, for $n\ge 2$, Player $\A$ selects
\begin{equation}
\label{eq:U_n}
U_n=\SIGMA(x_1,x_2,\dots,x_{n-1})
\end{equation}
  and 
Player $\B$ responds with 
\begin{equation}
\label{eq:7z}
x_n=\TAU(U_1,U_2,\dots,U_n).
\end{equation}
When all rounds are done,
we define
\begin{equation}
w_{\SIGMA,\TAU}=(U_1,x_1,U_2,x_2,\dots,U_n,x_n,\dots).
\end{equation}

(iii)
A strategy $\SIGMA$ for Player $\A$ is a {\em winning strategy} in $\OP(X,\mathcal{S})$ if 
Player $\A$  wins $w_{\SIGMA,\TAU} $
for every strategy $\TAU  $ for Player $\B$ in $\OP(X,\mathcal{S})$.  A strategy $\TAU $ for Player $\B$ is a {\em winning strategy} in $\OP(X,\mathcal{S})$ if 
Player $\B$ wins $w_{\SIGMA,\TAU} $
for every strategy $\SIGMA $ for Player $\A$ in $\OP(X,\mathcal{S})$.  

(iv) A strategy $\SIGMA $ 
for Player $\A$ is 
{\em stationary\/} if 
\begin{equation}
\label{eq:sigma:stationary}
\SIGMA (x_1,x_2,\dots,x_n)=\SIGMA (x_n)
\text{ for every }
(x_1,x_2,\dots,x_n)\in\Seq(X)\setminus \{\emptyset \}.
\end{equation}

(v) A  
strategy $\TAU  $ for Player $\B$ is {\em stationary\/} if 
\begin{equation}
\label{eq:tau:stationary}
\TAU (U_1,U_2,\dots,U_n)=\TAU  (U_n)
\text{ for every }
(U_1,U_2,\dots,U_n)\in \Seq(\mathcal{O}(X))\setminus\{\emptyset\}.
\end{equation}

(vi) The game $\OP(X,\mathcal{S})$
 is {\em determined\/} if either
$\A$ or
 $\B$ has a winning strategy in $\OP(X,\mathcal{S})$.
\end{definition}

\begin{lemma}\label{sequence:satisfies:S} 
If 
$\TAU$
is a stationary winning strategy for $\B$ in $\OP(X,\mathcal{S})$, then the sequence $\{\TAU(U_n):n\in \N\}$ satisfies $\mathcal{S}$ in $X$ for every sequence
$(U_1, U_2,\dots, U_n,\dots)\in \mathcal{O}(X)^\N$.
\end{lemma}

\begin{proof}
Let $(U_1, U_2,\dots, U_n,\dots)\in \mathcal{O}(X)^\N$ be a fixed infinite sequence.
Define 
the strategy $\SIGMA:\Seq(X)\to \mathcal{O}(X)$ for $\A$  in $\OP(X,\mathcal{S})$
by $\SIGMA(\emptyset)=U_1$ and 
$\SIGMA(x_1,\dots,x_n)=U_{n+1}$ for 
$(x_1,\dots,x_n)\in\Seq(X)\setminus\{\emptyset\}$.

Let $w_{\SIGMA,\TAU}$ be the play produced by 
strategies $\SIGMA$ and $\TAU$.
Since $\TAU$ is a winning strategy
for $\B$ in $\OP(X,\mathcal{S})$,
Player $\B$ wins $w_{\SIGMA,\TAU}$, which means that
the sequence $\{x_n:n\in\N\}$ satisfies property $\mathcal{S}$ in $X$. Since $\TAU$ is a stationary 
strategy,
it follows from \eqref{eq:7z} and \eqref{eq:tau:stationary} and 
that $x_n=\TAU(U_1,\dots,U_n)=\TAU(U_n)$.
\end{proof}

\begin{theorem}\label{stationary_winning_strategy}
For every space $X$, 
Player $\B$ has a stationary winning strategy in $\OP(X,\mathcal{S})$
if and only if $X$ has a dense subspace $D$ 
relatively satisfying $\mathcal{S}$ in $X$. 
\end{theorem}

\proof
Suppose that 
$\TAU$
is a stationary winning strategy for $\B$ in $\OP(X,\mathcal{S})$. 
Since $\TAU(U)\in U$ for every $U\in \mathcal{O}(X)$ by \eqref{eq:rule:of:the:game} and \eqref{eq:tau:stationary}, 
the set
$D=\{\TAU(U):U\in \mathcal{O}(X)\}$ 
is 
dense in $X$. 
It remains only to check that $D$ 
relatively satisfies $\mathcal{S}$ in $X$. 
Let $\{x_n:n\in\N\}$ be a sequence of points of $D$.
For every $n\in\N$, choose $U_n\in \mathcal{O}(X)$
such that $x_n=\TAU(U_n)$. Now Lemma  \ref{sequence:satisfies:S}
implies that the sequence $\{x_n:n\in\N\}$ satisfies $\mathcal{S}$
in $X$.
Therefore, $D$  relatively satisfies $\mathcal{S}$ in $X$.  

Next, suppose that $D$ is a dense subspace of $X$  relatively satisfying $\mathcal{S}$ in $X$. 
For every $U\in \mathcal{O}(X)$, the intersection 
$U\cap D$ is non-empty, so we can choose a point $d_U\in U\cap D$. 
Define the strategy
$\TAU: \Seq(\mathcal{O}(X))\setminus\{\emptyset\}\to X$
for $\B$ by 
\begin{equation}
\label{eq:def:tau:by:d}
\TAU(U_1,U_2,\dots,U_n)=d_{U_n}
\text{ for }
(U_1,U_2,\dots,U_n)\in \Seq(\mathcal{O}(X))\setminus\{\emptyset\}.
\end{equation}
It follows from \eqref{eq:tau:stationary}
and 
\eqref{eq:def:tau:by:d}
that
$\TAU$ is stationary. 

To show that $\TAU$ is a winning strategy for $\B$, 
let $\SIGMA$
be an arbitrary strategy for $\A$ in $\OP(X,\mathcal{S})$.
Let $w_{\SIGMA,\TAU}$ be the play produced by following strategies $\SIGMA$ and $\TAU$.
It follows from \eqref{eq:7z} and
\eqref{eq:def:tau:by:d}
that $x_n=\TAU(U_1,U_2,\dots,U_n)=d_{U_n}\in D$ for every $n\in\N$.
By hypothesis, the sequence 
$\{x_n:n\in\N\}$
satisfies $\mathcal{S}$ in $X$.
Hence, $\B$ wins the play $w_{\SIGMA,\TAU}$. Since $\SIGMA $ was an arbitrary strategy for $\A$, this means that $\TAU$ 
is a 
winning strategy in $\OP(X,\mathcal{S})$ for $\B$.
\endproof

\begin{pro}\label{O:stationary:wins}
Let $\mathcal{S}$ be a topological property of sequences such that:
\begin{itemize}
\item[(i)]
 for every space $X$
and each strictly increasing function $\varphi:\N\to\N$,
the inclusion $s\circ\varphi\in\mathcal{S}_X$ implies $s\in\mathcal{S}_X$; and 
\item[(ii)]
 for every space $X,$ every constant sequence belongs to $\mathcal{S}_X.$ 
\end{itemize}
If $X$ is not a \sel\ space, then $\A$ has a stationary winning strategy in $\OP(X,\mathcal{S}).$  
\end{pro}

\begin{proof}
Suppose that $X$ is not \sel. 
By Definition \ref{def:selectively:S}(iii),
there is a family $\{U_n:n\in\N\}$ of non-empty open subsets of $X$ such that if 
$x_n\in U_n$ for every $n\in \N$, then the sequence $\{x_n:n\in \N\}$ 
does not satisfy $\mathcal{S}$ in $X.$ 

Suppose that there is a point $x\in X$ such that $M=\{n\in\N:x\in U_n\}$ is infinite.
For every $n\in\N\setminus M$,
choose an arbitrary point $x_n\in U_n.$ 
Define $s\in X^{\N} $ by 
\begin{equation}
s(n)=
\left\{\begin{array}{ll}
x& \mbox{if } n\in M ; \\
x_n & \mbox{otherwise}.\\
\end{array}
\right.
\end{equation}
Let $\varphi:\N\to M$ be an order preserving bijection. Since $s\circ\varphi$ is the constant sequence, $s\circ\varphi\in\mathcal{S}_X$  by (ii). By (i), $s$ satisfies $\mathcal{S}$ in $X.$ On the other hand, since $s(n)\in U_n$ for every $n\in\N,$ the sequence $s$ 
does not satisfies $\mathcal{S}$ in $X.$ 
This contradiction shows that
the set $M_x=\{n\in\N:x\in U_n\}$ is finite for every $x\in X$,
so we can
let $m(x)=\sup M_x+1$ if $M_x\not=\emptyset$ and $m(x)=1$ if $M_x=\emptyset.$ 

Define the strategy
$\SIGMA:\Seq(X)\to \mathcal{O}(X)$
for $\A$ in $\OP(X,\mathcal{S})$ by $\SIGMA(\emptyset)=U_1$ and
\begin{equation}
\label{eq:12:f}
\SIGMA(x_1,x_2,\dots,x_n)=
U_{m(x_n)}
\text{ for }
(x_1,x_2,\dots,x_n)\in\Seq(X)\setminus\{\emptyset\}.
\end{equation} 
It is clear from \eqref{eq:sigma:stationary} and \eqref{eq:12:f}
that $\SIGMA$ is stationary.  

To prove that $\SIGMA $ is a winning strategy for $\A$ in 
$\OP(X,\mathcal{S})$, assume 
that $\TAU$
is an arbitrary a strategy for $\B$ in $\OP(X,\mathcal{S})$. 
Let $w_{\SIGMA,\TAU}$ be the play produced by following strategies $\SIGMA$ and $\TAU$.
Since $x_1\in U_1$ and $x_{n}\in U_{m(x_{n-1})}$ for every $n\geq 2$
by \eqref{eq:rule:of:the:game} and \eqref{eq:7z}, the function $\varphi:\N\to\N$ given by $\varphi(1)=1$ and 
$ \varphi(n)=m(x_{n-1})$ for every $n\geq 2,$ is strictly increasing. For every $n\in\N\setminus \varphi(\N)$
choose an arbitrary point $y_n\in U_n.$ (This can be done as $U_n$ is non-empty.) Define $s\in X^\N$ by 
\begin{equation}
s(n)=
\left\{\begin{array}{ll}
x_{\varphi^{-1}(n)} & \mbox{if } n\in\varphi(\N) ; \\
y_n & \mbox{otherwise}.\\
\end{array}
\right.
\end{equation}
Since $s(n)\in U_n$ for every $n\in \N, $
$s\not\in\mathcal{S}_X.$ By hypothesis, $s\circ\varphi\not\in\mathcal{S}_X.$
 Therefore,
the sequence $\{x_n:n\in \N\}=s\circ\varphi$ does not satisfy $\mathcal{S}$ in $X$, so
$\A$ wins the play $w_{\SIGMA ,\TAU}$. Since $\TAU $ was an arbitrary strategy for $\B$ in $\OP(X,\mathcal{S})$, $\SIGMA $ is a winning strategy for $\A$ in $\OP(X,\mathcal{S})$. 
\end{proof}

Diagram~3
describes a fine structure of the area represented by arrow (c) in Diagram~2 by 
collecting
the implications 
that hold 
for every topological property of sequences $\mathcal{S}$. Arrow $(c_4)$ holds for those properties $\mathcal{S}$ that satisfy the hypothesis of Proposition  \ref{O:stationary:wins}.
(Note that all five properties from Example \ref{examples} satisfy this hypothesis.) 

The following proposition easily follows from Definitions \ref{comparing:strength:of:R:and:S} and \ref{def:game:OP(X,S)}(iii).

\begin{pro}
\label{winning:strategies:for:R:and:S}
If $\mathcal{R}$ and
$\mathcal{S}$ are 
 topological properties of sequences such that $\mathcal{R}$ is stronger than
$\mathcal{S}$, then:
\begin{itemize}
\item[(i)] every winning strategy for $\A$ in $\OP(X,\mathcal{S})$ is also a winning strategy for $\A$ in
$\OP(X,\mathcal{R})$;
\item[(ii)] every winning strategy for $\B$ in $\OP(X,\mathcal{R})$ is also a winning strategy for $\B$ in
$\OP(X,\mathcal{S})$.
\end{itemize}
\end{pro}

{\small
\begin{center} 
\medskip\hspace{1em}
\xymatrix{
X \text{ has a dense subspace that relatively satisfies } \mathcal{S}\ar[d] \text{ in }X\\
\text{Player } 
\B \text{ has a stationary winning strategy in } \OP(X,\mathcal{S})\ar[d]_{(c_1)}\ar[u]\\
\text{Player } 
\B \text{ has a winning strategy in } \OP(X,\mathcal{S})\ar[d]_{(c_2)}\\
\text{Player } 
\A \text{ does not have a winning strategy in } \OP(X,\mathcal{S})\ar[d]_{(c_3)}\\  
\text{Player } 
\A \text{ does not have a stationary winning strategy in } \OP(X,\mathcal{S})\ar[d]_{(c_4)}\\  
X \text{ is \sel}
}
\end{center}
\smallskip
\begin{center}
Diagram 3.
\end{center}
}

\section{Special cases $\Ssp(X)$
and $\Sp(X)$ of the open-point game $\OP(X,\mathcal{S})$}
\label{sec:5}

Two special cases of the game $\OP(X,\mathcal{S})$ 
play a prominent role in this paper.

\begin{definition}
\label{def:Ssp:Sp}
(i) When the topological property of sequences $\mathcal{S}$
is defined as 
in item (i) of Example \ref{examples}, 
we shall denote the game $\OP(X,\mathcal{S})$ simply by $\Ssp(X)$
and call it
the {\em selectively sequentially pseudocompact game on $X$\/}.

(ii)
When the topological property of sequences $\mathcal{S}$
is defined as 
in item (v) of Example \ref{examples}, 
we shall denote the game $\OP(X,\mathcal{S})$ simply by $\Sp(X)$
and call it the 
{\em selectively pseudocompact game on $X$\/}.
\end{definition}

The abbreviations and terms in Definition 
\ref{def:Ssp:Sp} are selected 
in such a way that
to remind the reader that \sel\ spaces in 
items (i) and (ii)  
are precisely the selectively sequentially pseudocompact 
and 
selectively pseudocompact spaces, respectively; see items (i) and (iii) 
of
Remark \ref{names:remark}. 

The next theorem gives an internal characterization of  spaces $X$ such that Player $\B$ has a stationary winning strategy in the games $\Ssp(X)$ and $\Sp(X)$,
respectively.

\begin{theorem}
\label{two:dense:corollaries}
Let $X$ be a topological space.
\begin{itemize}
\item[(i)] Player $\B$ has a stationary winning strategy in $\Ssp(X)$
if and only if $X$ has a dense subspace $D$ which is relatively sequentially compact in $X$; that is, every sequence of points of $D$ has a subsequence which converges to some point of $X$.
\item[(ii)] Player $\B$ has a stationary winning strategy in $\Sp(X)$
if and only if $X$ has a dense subspace $D$ which is relatively countably compact in $X$; that is, every sequence of points of $D$ has an accumulation point in $X$.
\end{itemize}
\end{theorem}
\begin{proof}
Item (i) follows from Remark \ref{names:remark}(i) 
and Theorem \ref{stationary_winning_strategy}
applied to the property $\mathcal{S}$ from item (i) of 
Example \ref{examples},  
and item (ii)
follows from Remark \ref{names:remark}(iii) 
and Theorem \ref{stationary_winning_strategy}
applied to the property $\mathcal{S}$ from item (v) of 
Example \ref{examples}.
\end{proof}

Since every dyadic space has a dense sequentially compact subspace,
from Theorem \ref{stationary_winning_strategy}(i) we obtain the following
corollary strengthening \cite[Corollary 4.6]{DoS}.
\begin{cor}
\label{dyadic:result}
For every dyadic space $X$ (in particular, for every compact group $X$), Player $\B$ has a stationary winning strategy in the \ssp\ game $\Ssp(X)$ on $X$.
\end{cor}

Let $\mathcal{S}$ be a topological property of sequences which is weaker than  item (i) but stronger than 
item (v) of Example \ref{examples}.
Recalling Remark \ref{decreasing:strength} and using our new notations, we obtain Diagram~4. 

\begin{sidewaysfigure}
{\tiny
\begin{center} \medskip
\vskip500pt
\xymatrix{
X \text{ is \seqc}\ar[d]_{1}\ar[r]&
X \text{ satisfies } \mathcal{S}\ar[d]_{7}\ar[r]& 
X \text{ is \cc} \ar[d]_{13}\\
X \text{ has a dense \seqc\ subspace}\ar[r]\ar[d]_{2}&
X \text{ has a dense subspace } Y \text{ that satisfies } \mathcal{S}\ar[d]_{8}\ar[r]& 
X \text{ has a dense \cc\ subspace}\ar[d]_{14}\\
X \text{ has a dense relatively \seqc\ subspace}\ar[d]\ar[r]& 
X \text{ has a dense subspace } Y \text{ that relatively satisfies } \mathcal{S}\ar[d]\ar[r]& 
X \text{ has a dense relatively \cc\ subspace}\ar[d]\\
\B \text{ has a stationary winning strategy in } \Ssp(X)\ar[d]_{3}\ar[u]\ar[r]&
\B \text{ has a stationary winning strategy in } \OP(X,\mathcal{S})\ar[d]_{9}\ar[u]\ar[r]&
\B \text{ has a stationary winning strategy in } \Sp(X)\ar[d]_{15}\ar[u]\\
\B \text{ has a winning strategy in } \Ssp(X)\ar[d]_{4}\ar[r]&
\B \text{ has a winning strategy in } \OP(X,\mathcal{S})\ar[d]_{10}\ar[r]&
\B \text{ has a winning strategy in } \Sp(X)\ar[d]_{16}\\
\A \text{ does not have a winning strategy in } \Ssp(X)\ar[d]_{5}\ar[r]&
\A \text{ does not have a winning strategy in } \OP(X,\mathcal{S})\ar[d]_{11}\ar[r]&
\A \text{ does not have a winning strategy in } \Sp(X)\ar[d]_{17}\\
\A \text{ does not have a stationary winning strategy in } \Ssp(X)\ar[d]_{6}\ar[r]&
\A \text{ does not have a stationary winning strategy in } \OP(X,\mathcal{S})\ar[d]_{12}\ar[r]&
\A \text{ does not have a stationary winning strategy in } \Sp(X)\ar[d]_{18}\\    
X \text{ is \ssp}\ar[r]& 
X \text{ is \sel}\ar[r]& 
X \text{ is \sp}\ar[d]_{19}\\
& & X \text{ is pseudocompact}.
}

Diagram 4.
\end{center}
}
\end{sidewaysfigure}

The Stone-\v{C}ech compactification of the natural numbers
is a compact space which is not
\ssp\ \cite[Example 2.5]{DoS}. Hence, none of the properties in the right column of Diagram 4 imply any of the properties in the middle column of the diagram, 
and none of the properties in the middle column of Diagram 4
imply any of the properties in the left column  
of this diagram. 

Example \ref{plank} shows that 
arrows 1, 7 and 13 of Diagram 4 are not reversible. Example \ref{mrowka} shows that arrows 2, 8 and 14 are not reversible. 
Corollary \ref{cor:5.11} shows that arrows 3, 9 and 15 are not reversible.
Corollary \ref{not:arrow:c4} shows that arrows 6 and 18 are not reversible, and arrow 12 is not reversible under additional assumption that the property $\mathcal{S}$ is closed under projections. The non-reversibility of arrow 19 is established
in \cite{GO,GT}.
The reversibility of the remaining numbered arrows in Diagram 4 remains unclear; see Questions~\ref{que:7.1} and~\ref{question:new}.

\section{Producing stationary winning strategies for Player $\A$ in $\OP(X,\mathcal{S})$ from non-stationary ones}

For a set $Y$, we use $\alpha(Y_{\disc})$ to denote the one point compactification of $Y_{\disc},$ where 
$Y_{\disc}$ is the set $Y$ endowed with the discrete topology. 

In the following theorem, we describe a general technique which employs a winning strategy for Player $\A$ in $\OP(X,\mathcal{S})$ to produce a {\em stationary\/} winning strategy 
for Player $\A$ in the game $\OP(X\times \alpha(\Seq(X)_{\disc}),\mathcal{S})$ played on the product of $X$ with the one point compactification of the discrete space $\Seq(X)_{\disc}$.

\begin{theorem}\label{from:winning:to:stationary}
Let $\mathcal{S}$ be a topological property of sequences preserved by projections.\footnote{This means that if a sequence
$\{(x_n,y_n):n\in\N\}$ of points of a product $X\times Y$ satisfies $\mathcal{S}$ in $X\times Y$, then the sequence $\{x_n:n\in\N\}$ must satisfy $\mathcal{S}$ in $X$.}
If $X$ is a space such that 
Player $\A$ has a winning strategy in $\OP(X,\mathcal{S})$, then 
Player $\A$ has a stationary winning 
strategy in $\OP(X\times \alpha(\Seq(X)_{\disc}),\mathcal{S}).$
\end{theorem}

\begin{proof}
Let $Y=\alpha(\Seq(X)_{\disc}).$ 
For $n\in\N$, $\sigma=(x_1,x_2,\dots,x_n)\in \Seq(X)$ and $x\in X$, we use $\sigma^\wedge x$ to denote the sequence
$(x_1,x_2,\dots,x_n,x)\in\Seq(X)$ of length $n+1$. For the empty sequence $\emptyset\in \Seq(X)$ and $x\in X$, we let $\emptyset^\wedge x$ to be the sequence $(x)\in \Seq(X)$ of length $1$.

Let $\SIGMA$ be a winning strategy for Player $\A$ in $\OP(X,\mathcal{S})$. Define the map $\SIGMA':\Seq(X\times Y)
\to \mathcal{O}(X\times Y)$ by 
\begin{equation}
\label{eq:14:new}
\SIGMA'(\emptyset)= \SIGMA(\emptyset) \times\{\emptyset\}
\end{equation}
  and 
\begin{equation}
\label{eq:14:g}
\SIGMA'((x_1,y_1),(x_2,y_2),\ldots,(x_n,y_n))=\left\{\begin{array}{ll}
\SIGMA(y_n{}^\wedge x_n) \times \{ y_n{}^\wedge x_n\}& \mbox{if } y_n\in\Seq(X); \\
\SIGMA(\emptyset) \times \{\emptyset\}& \mbox{otherwise}\\
\end{array}
\right. 
\end{equation}
for $((x_1,y_1),(x_2,y_2),\ldots,(x_n,y_n))\in\Seq(X\times Y)\setminus\{\emptyset\}$.
If $x\in X$ and
$y\in \Seq(X)$, 
then $y^\wedge x\in\Seq(X)$, so the singleton $\{y^\wedge x\}$ is an open subset of $Y$.
Similarly, $\emptyset\in\Seq(X)$, so the singleton $\{\emptyset\}$ is an open subset of $Y$ as well.
This shows that the map $\SIGMA'$ is well defined.
By Definition~\ref{def:game:OP(X,S)}(i), $\SIGMA'$ is a strategy 
for Player $\A$ in the game $\OP(X\times Y,\mathcal{S})$.
By \eqref{eq:tau:stationary} and \eqref{eq:14:g}, the strategy $\SIGMA'$ is stationary. 

We are going to show that $\SIGMA'$ is a winning strategy for Player $\A$ in the game $\OP(X\times Y,\mathcal{S})$.
By Definition \ref{def:game:OP(X,S)}(iii), to show this we need to consider an arbitrary strategy 
$\TAU'$ for Player $\B$ in $\OP(X\times Y,\mathcal{S})$
and show that Player $\A$ wins the game
\begin{equation}
\label{eq:15:l}
w_{\SIGMA',\TAU'}=(W_1,(x_1,y_1),W_2,(x_2,y_2),\ldots,W_n,(x_n,y_n),\ldots)
\end{equation}
produced by following the strategies $\SIGMA'$ and $\TAU'$. 

\begin{claim}
\label{first:claim}
\begin{itemize}
\item[(i)]
$x_n\in\SIGMA(y_n)$ for every $n\in\N$.
\item[(ii)]
$y_n=(x_1,x_2,\dots,x_{n-1})\in \Seq(X)$ for every $n\in\N$,
where we consider $(x_1,x_2,\dots,x_{n-1})$ to be the empty sequence $\emptyset$ for $n=1$.
\end{itemize}
\end{claim}
\begin{proof}
We prove this claim by induction on $n\in\N$.

\smallskip
{\em Basis of induction\/}.
Note that
$W_1=\SIGMA'(\emptyset)$ 
and
$(x_1,y_1)=\TAU'(W_1)$
by \eqref{eq:15:l} and 
Definition \ref{def:game:OP(X,S)}(ii).
Combining the first equation with \eqref{eq:14:new}, we get 
$W_1=\SIGMA(\emptyset) \times\{\emptyset\}$.
Similarly,
$(x_1,y_1)=\TAU'(W_1)\in W_1$
by \eqref{eq:15:l} and 
Definition \ref{def:game:OP(X,S)}(i).
Therefore,
$x_1\in\SIGMA(\emptyset)$
and $y_1=\emptyset$,
which implies
$x_1\in \SIGMA(y_1).$ 

\smallskip
{\em Inductive step\/}.
Let $n\in\N$ and $n\ge 2$. Since $\SIGMA'$ is stationary, 
\begin{equation}
\label{eq:17:s}
W_{n}=\SIGMA'(x_{n-1},y_{n-1})=\SIGMA(y_{n-1}{}^\wedge x_{n-1}) \times \{ y_{n-1}{}^\wedge x_{n-1}\}
\end{equation}
by \eqref{eq:14:g}, \eqref{eq:15:l} and 
Definition \ref{def:game:OP(X,S)}(ii).
Since 
$(x_n,y_n)=\TAU'(W_1,W_2,\ldots,W_{n})\in W_n$
by \eqref{eq:15:l} and 
Definition \ref{def:game:OP(X,S)}(i),
combining this with \eqref{eq:17:s} yields
$y_n=y_{n-1}{}^\wedge x_{n-1}$
and 
$x_n\in \SIGMA(y_{n-1}{}^\wedge x_{n-1})$.
Since $y_{n-1}=(x_1,x_2,\dots,x_{n-2})$ by our inductive assumption,
we get $y_n=(x_1,x_2,\dots,x_{n-2})^\wedge x_{n-1}=
(x_1,x_2,\dots,x_{n-1})\in\Seq(X)$ and 
$x_n\in \SIGMA(y_n)$.
\end{proof}

For every $V\in\mathcal{O}(X)$, select $a_V\in V$ arbitrarily.
Define $\TAU:\Seq(\mathcal{O}(X))\setminus\{\emptyset\}\to X$ by
\begin{equation}
\label{eq:15:g}
\TAU(V_1,\ldots,V_n)=\left\{\begin{array}{ll}
x_n& \mbox{if } V_n=\SIGMA(y_n); \\
a_{{}_{V_n}}& \mbox{otherwise}\\
\end{array}
\right. 
\hskip30pt
\text{for }(V_1,V_2,\dots,V_n)\in\Seq(\mathcal{O})\setminus\{\emptyset\}.
\end{equation}

It follows from Claim \ref{first:claim}(i) that
 $\TAU$ is a well-defined strategy for 
Player $\B$ in $\Sp(X,\mathcal{S})$.
Let
\begin{equation}
\label{eq:16:g}
w_{\SIGMA,\TAU}=(U_1,z_1,U_2,z_2,\dots,U_n,z_n,\dots)
\end{equation}
be the game produced by following the strategies $\SIGMA$ and $\TAU$. 

\begin{claim}
\label{second:claim}
$U_n=\SIGMA(y_n)$ and $z_n=x_n$ for all $n\in\N$. 
\end{claim}
\begin{proof}
We prove this claim by induction on $n\in\N$.

\smallskip
{\em Basis of induction\/}.
Recall that $y_1=\emptyset$ by Claim \ref{first:claim}(ii).
It follows from \eqref{eq:16:g} and 
Definition \ref{def:game:OP(X,S)}(ii) that
$U_1=\SIGMA(\emptyset)=\SIGMA(y_1)$.
Now $z_1=\TAU(U_1)=\TAU(\SIGMA(y_1))=x_1$
by Definition \ref{def:game:OP(X,S)}(ii) and \eqref{eq:15:g}.

\smallskip
{\em Inductive step\/}.
Let $n\in\N$ and $n\ge 2$.
Suppose that 
$U_i=\SIGMA(y_i)$ and $x_i=z_i$ for every $i< n.$
Then
$U_{n}=\SIGMA(z_1,z_2,\ldots,z_{n-1})
=
\SIGMA(x_1,x_2,\ldots,x_{n-1})
=
\SIGMA(y_n)
$ 
by \eqref{eq:16:g},
Definition \ref{def:game:OP(X,S)}(ii), 
our inductive assumption and Claim \ref{first:claim}(ii).
Therefore, $z_{n}=\TAU(U_1,U_2,\ldots,U_{n})=x_n$
by
Definition \ref{def:game:OP(X,S)}(ii)
and
\eqref{eq:15:g}. 
\end{proof}

By Claim \ref{second:claim} and  (\ref{eq:16:g}), we get
$w_{\SIGMA,\TAU}=(U_1,x_1,U_2,x_2,\dots,U_n,x_n,\dots)$.
Since $\SIGMA$ is a winning strategy for Player $\A$ in $\OP(X,\mathcal{S})$ and $\TAU$ is a strategy for Player $\B$ in $\OP(X,\mathcal{S})$,
Player $\A$
wins the game $w_{\SIGMA,\TAU}$. This means that
the sequence
$\{x_n:n\in \N\}$ does not satisfy $\mathcal{S}$ in $X.$ Since $\mathcal{S}$ is preserved by projections, 
the sequence  $\{(x_n,y_n):n\in \N\}$ does not satisfy $\mathcal{S}$ in $X\times Y.$ Therefore, Player $\A$ wins the 
game \eqref{eq:15:l}.
\end{proof}

\section{``Injective version'' of a van Douwen MAD family on an arbitrary set} 

\label{MAD:family}

In our construction of an example in the next section, we shall need the following set-theoretic result of independent interest. When $D=\N$, this result becomes an ``injective version'' of a van Douwen MAD  family 
constructed by D. Raghavan in \cite[Theorem 2.14]{R}.

\begin{theorem}
\label{set-theoretic:result}
Let $D$ be an infinite set and let $\mathscr{I}(D)$ denote the family of all injective functions $g$ from a countably infinite subset $\dom (g)$ of $\N$ to $D$.
Then 
there exists a 
family $\mathscr{F}\subseteq \mathscr{I}(D)$ having two properties:
\begin{itemize}
\item[(A)] If $f,g\in\mathscr{F}$ are distinct, then the set $\{n\in\dom (f)\cap \dom (g): f(n)=g(n)\}$ is finite.
\item[(B)] For every $g\in\mathscr{I}(D)$, there exists $f\in\mathscr{F}$ such that $\{n\in\dom (f)\cap \dom (g): f(n)=g(n)\}$ is infinite.
\end{itemize}
\end{theorem}
\begin{proof}
A family $\mathscr{G}\subseteq \N^{\N}$ is said to be \emph{almost disjoint\/} if 
the set $\{n\in \N:f(n)=g(n)\} $ is finite whenever $f,g\in \mathscr{G}$ are distinct.
Following \cite[Definition 1.3]{R},
we shall say that $p$ is an \emph{infinite partial function} 
if $p\in \N^P$ 
for some infinite subset $P$ of $\N$.
An almost disjoint family $\mathscr{G}\subseteq \N^{\N}$ is a 
\emph{van Douwen MAD family\/} if for every infinite partial function $p\in\N^P$, there is $g\in \mathscr{G}$ such that 
the set $\{n\in P:p(n)=g(n)\} $ is infinite \cite[Definition 1.4]{R}.

We fix a van Douwen MAD family
 $\mathscr{G}$ of size $\mathfrak{c}$;
the existence of such a family was proved by 
D. Raghavan
\cite[Theorem 2.14]{R}.
For every $g\in \mathscr{G}$,
let 
\begin{equation}
\label{eq:I:g}
\mathcal{I}_g=\{A\subseteq \N: A
\text{ is infinite and }
g\restriction_A
\text{ is an injection}\}.
\end{equation}
Define 
$$\mathscr{G}'=\{g\in\mathscr{F}: \mathcal{I}_g\not=\emptyset\}.
$$
For every $g\in\mathscr{G}'$, 
use Zorn's lemma to fix a 
maximal almost disjoint subfamily
$\mathcal{A}_g$ of $\mathcal{I}_g$; that is,
\begin{itemize}
\item[(a)] $A\cap A'$ is finite whenever $A,A'\in\mathcal{A}_g$ are distinct;
\item[(b)] if $T\in\mathcal{I}_g$, then $T\cap A$ is infinite for some $A\in\mathcal{A}_g$.
\end{itemize}
Clearly, $\mathcal{A}_g\not=\emptyset$.

Applying Zorn's lemma, we can 
fix a family $\mathscr{H}\subseteq D^\N$ having the following properties:
\begin{itemize}
\item[(i)] each $h\in\mathscr{H}$ is injective;
\item[(ii)] if $h_1,h_2\in\mathscr{H}$ and $h_1\not=h_2$, then the set 
$h_1(\N)\cap h_2(\N)$ is finite;
\item[(iii)] if $S$ is an infinite subset of $D$, then $S\cap h(\N)$ is infinite for some $h\in\mathscr{H}$.
\end{itemize}

If $g\in\mathscr{G}', h\in\mathscr{H}, A\in\mathcal{A}_g$, define $f_{g,h,A}=h\circ g\restriction_A$ 
and consider the family $$\mathscr{F}=\{f_{g,h,A}:g\in\mathscr{G}', h\in\mathscr{H}, A\in\mathcal{A}_g\}$$

\begin{claim}
\label{claim:injectiove:restriction}
$\mathscr{F}\subseteq \mathscr{I}(D)$.
\end{claim}
\begin{proof}
This follows from $A\in\mathcal{A}_g\subseteq \mathcal{I}_g$,
\eqref{eq:I:g} and item (i) of the definition of $\mathscr{H}.$ 
\end{proof}

\begin{claim}\label{almost:disjoint} 
If $f_1,f_2$ are different elements of $\mathscr{F}$, then the set $\{m\in \dom(f_1)\cap \dom(f_2):f_1(m)= f_2(m)\}$  is finite.
\end{claim}
\begin{proof}
Let $g_i\in\mathscr{G}', h_i\in\mathscr{H}, A_i\in\mathcal{A}_{g_i}$, 
such that $f_i=h_i\circ g_i\restriction_{A_i}$ for $i=1,2$.  

\smallskip
{\em Case 1\/}. {\sl $h_1\not =h_2$.\/} 
From item (ii) of the definition of $\mathscr{H}$, we conclude that
the set $h_1(\N)\cap h_2(\N)$ is finite.
Since $h_1$ is injective by item (i) of
the definition of $\mathscr{H}$, the 
subset $h_1^{-1}(h_1(\N)\cap h_2(\N))$ of $\N$ is finite, so 
$h_1^{-1}(h_1(\N)\cap h_2(\N))\subseteq n_0$ for some $n_0\in\N$.
Since $g_1\restriction_{A_1}$ is an injection, 
we can take $n_1\in \N $ such that $g_1(m)>n_0$ for every $m\in A_1$ with $m>n_1.$ 

Take $m\in\N$ such that $f_1(m)= f_2(m).$ Then $m\in $ dom$(f_1)=A_1$ and 
$h_1( g_1(m))=f_1(m)=f_2(m)= h_2(g_2(m)).$
Therefore $g_1(m)\in h_1^{-1}(h_1(\N)\cap h_2(\N))\subseteq n_0.$
Hence, $g_1(m)< n_0,$ since $m\in A_1, m\leq n_1.$ 
Then the set $\{m\in \N:f_1(m)\in f_2(\N)\}$ is finite. 

\smallskip
{\em Case 2\/}. 
{\sl $h_1=h_2=h$  and  $g_1\not = g_2.$\/}
Since $g_1,g_2\in \mathscr{G}'\subseteq \mathscr{G}$
are distinct and the family $\mathscr{G}$ is almost disjoint,
the set $\{k\in\N: g_1(k)=g_2(k)\}$ is finite, so we can fix 
$n\in \N $ such that $g_1(m)\not = g_2(m)$ whenever $m\in\N$ and $m>n$.
Since 
$h$
is injective by item (i) of the definition of $\mathscr{H}$,  we have
\begin{equation}
\label{eq:5:f}
f_1(m)=h\circ g_1(m)\not = h\circ g_2(m)=f_2(m)
 \text{ for every }
m>n. 
\end{equation}
Then the set $\{m\in \N:f_1(m)=f_2(m)\}$ is finite.

\smallskip
{\em Case 3\/}. {\sl $h_1=h_2=h,g_1=g_2=g$ and $ A_1\not =A_2.$\/}
Since $A_i\in\mathcal{A}_{g_i}=\mathcal{A}_g$ for $i=1,2$,
the set $A_1\cap A_2$ is finite by item (a) of the definition of $\mathcal{A}_g$.
Hence,
$A_1\cap A_2\subseteq  n$ for some $n\in\N$. 
Therefore $\{m\in \N:f_1(m)=f_2(m)\}\subseteq A_1\cap A_2\subseteq n $. 
\end{proof}

\begin{claim}\label{van:douwen}
For every injective function $p\in D^P$, where $P$ is an infinite subset of $\N$, there is $f\in \mathscr{F}$ such that 
the set $\{n\in P:p(n)=f(n)\} $ is infinite.
\end{claim}
\begin{proof}
Since $p$ is injective, $p(P)$ is infinite. 
By item (iii) of the definition of $\mathscr{H},$ there is  $h\in\mathscr{H}$ such that 
the set $B=\{n\in P:p(n)\in h(\N)\}$ is infinite. Define $q\in\N^{B}$ by $q(n)=h^{-1}(p(n))$ for every $n\in B.$ Since $h$ is injective, $q$ is well defined. Since $\mathscr{G}$ is a van Douwen MAD family, there 
is $g\in\mathscr{G}$ such that the set $C=\{n\in B:q(n)=g(n)\}$ is infinite. Since $p$  is
 injective, $g$ is injective in $C.$ By maximality of $\mathcal{A}_g$, there is $A\in \mathcal{A}_g$ such that 
$A\cap C$ is infinite. Therefore $f=f_{g,h,A}\in \mathscr{F}.$
Take $n\in A\cap C.$ Then $f(n)=h(g(n))=h(q(n))=h( h^{-1}(p(n)))=p(n).$ Therefore
$A\cap C\subseteq \{n\in P:p(n)=f(n)\}.$
\end{proof}

Item (A) is proved in  Claim \ref{almost:disjoint}, and item (B) is proved in 
Claim \ref{van:douwen}.
\end{proof}

\section{Example showing that arrow $(c_1)$ of Diagram 3
is not reversible}
\label{arrow:c_1}

Berner gave an example of a pseudocompact space without a dense relatively countably compact subspace in \cite[Section 5]{B}. The space from our next theorem is a quite significant modification of Berner's example based on the family $\mathscr{F}$ constructed in Theorem \ref{set-theoretic:result}.

\begin{theorem}
\label{arrow:3}
There exists 
a locally compact, first-countable, zero-dimensional space $X$ such that  Player $\B$ has a winning strategy in $\Ssp(X)$ but 
does not have a stationary winning strategy even  in $\Sp(X).$
\end{theorem}
\begin{proof}
Let $C$ be 
the Cantor set.
For every $c\in C$,
fix a strictly decreasing base $\{W_n^c:n\in \N\}$ at $c$ consisting of clopen subsets of $C$ such that $W_0^c=C$,
and let
\begin{equation}
\label{Vnc}
V_n^c=W_n^c\setminus W_{n+1}^c
\text{ for every }n\in \N.
\end{equation}

The following claim is immediate from this definition:
\begin{claim}
\label{claim:disjoint:partition}
For every $c\in C$, the family
$\{V_n^c:n\in\N\}$ is a partition of $C\setminus\{c\}$ consisting of non-empty clopen subsets of $C$. 
\end{claim}

Let $ D$ be 
a set
of cardinality $\mathfrak{c}^+$.
Consider the discrete topology on $D$, and let 
$C\times D$ be equipped with the Tychonoff product topology. 
For $(c,f)\in C\times \mathscr{F}$
and
$n\in $ dom$(f)$, both
\begin{equation}
\label{eq:M}
M_{c,f}^n=V_n^c\times \{f(n)\}
\end{equation}
and 
\begin{equation}
\label{eq:0}
O_{c,f}^n=\bigcup_{m\in \textrm{ dom}(f), m> n} M_{c,f}^m
\end{equation}
are clopen subsets of  $C\times D$.

\begin{claim}
\label{disjoint:claim}
If $(c_1,f_1), (c_2,f_2)\in C\times \mathscr{F}$
are distinct, then 
\begin{equation}
\label{eq:disjoint:O}
O_{c_1,f_1}^n\cap O_{c_2,f_2}^{n}=\emptyset
\end{equation}
for some $n\in\N$.
\end{claim}
\begin{proof}
{\em Case 1\/}. {\sl $c_1\not =c_2$\/}. 
Since $\{W_n^{c_i}:n\in \N\}$ is a strictly decreasing base  at $c_i$ for $i=1,2$,
there exists $n\in \N$ such that $W_n^{c_1}\cap W_n^{c_2}=\emptyset$. Moreover,
\eqref{Vnc} implies that
$\bigcup_{m>n} V_m^{c_i}\subseteq W_n^{c_i}$ for $i=1,2$. 
Combining this with \eqref{eq:M} and \eqref{eq:0}, we get
\eqref{eq:disjoint:O}.

{\em Case 2\/}. {\sl $c_1=c_2,f_1\not =f_2$.\/} 
By Claim \ref{almost:disjoint}, there is $n\in\N$ such that $\{m\in \N:f_1(m)= f_2(m)\}\subseteq n.$
Suppose that \eqref{eq:disjoint:O} is not satisfied. By \eqref{eq:0}, there exists $m_i\in A_i$ satisfying $m_i>n$ for $i=1,2$ such that 
$M_{c,f_1}^{m_1}\cap M_{c,f_2}^{m_2}\not=\emptyset$.
Since $M_{c,f_i}^{m_i}=V_{m_i}^c\times \{f_i(m_i)\}$
for $i=1,2$ by \eqref{eq:M},
we deduce that
$V_{m_1}^c\cap V_{m_2}^c\not=\emptyset$ and
$f_1(m_1)=f_2(m_2)$.
From the former inequality and Claim \ref{claim:disjoint:partition}, 
we conclude that $m_1=m_2=m$, so from the latter equality we get
$f_1(m)=f_2(m)$. Hence $m< n,$ which is a contradiction. Therefore \eqref{eq:disjoint:O} is satisfied.
\end{proof}

Without lost of generality we can assume that $\mathscr{F}\cap D=\emptyset$. Consider
the topology on 
the set
$$
X=(C\times \mathscr{F})\cup (C\times D)
$$
defined by declaring
$C\times D$ to be 
an open subspace of $X$ and taking the family
\begin{equation}
\label{eq:B}
\mathcal{B}_{c,f}=\{B_{c,f}^n:n\in \N\},
\end{equation}
where
\begin{equation}
\label{eq:B:element}
B_{c,f}^n=\{(c,f)\}\cup O_{c,f}^n
\ \ \text{ for }
n\in\N,
\end{equation}
as a local base at each point $(c,f)\in C\times \mathscr{F}$.

Clearly, $C\times D$ is dense in $X$ and $C\times \mathscr{F}$ is a closed discrete subspace of $X$.
The fact that $X$ is first countable is straightforward from the definition. 

\begin{claim}
\label{Hausdorff:claim}
$X$ is Hausdorff. 
\end{claim}
\begin{proof}
Let 
 $(c_1,f_1), (c_2,f_2)\in C\times \mathscr{F}$ 
be distinct.
Let $n\in\N$ be as in the conclusion of Claim \ref{disjoint:claim}. Since $\mathscr{F}\cap D$ is empty,
it follows from \eqref{eq:disjoint:O} and \eqref{eq:B:element}
that 
$B_{c_1,f_1}^n\cap B_{c_2,f_2}^{n}=\emptyset$.
Since $B_{c_i,f_i}^n\in \mathcal{B}_{c_i,f_i}$ for $i=1,2$
by \eqref{eq:B},
the points $(c_1,f_1)$ and $(c_2,f_2)$ of $X$
can be separated by two disjoint basic open subsets of $X$. 
Next, assume that
$(c_0,d)\in C\times D$ and $(c,f)\in C\times \mathscr{F}$. 
Since $f\in\mathscr{F}\subseteq \mathscr{I}(D)$, $f$ is an injection by definition of $\mathscr{I}(D)$,
so 
  there is at most one $n\in \N$ such that 
$f(n)=d.$ If $m>n$, 
$(C\times\{d\})\cap M_{c,f}^m=\emptyset$
by \eqref{eq:M}.
Combining this with \eqref{eq:0},
we get $(C\times\{d\})\cap O_{c,f}^n=\emptyset$,
so
$(C\times\{d\})\cap B_{c,f}^n=\emptyset$
by \eqref{eq:B:element}.
It follows from \eqref{eq:B} and \eqref{eq:B:element} that 
$B_{c,f}^n$ is an open neighborhood of 
$(c,f)$. 
Since $C\times \{d\}$ is open in $C\times D$ and $C\times D$ is open in $X$, the set $C\times \{d\}$ is open in $X$.
We conclude that  $C\times \{d\}$ and $B_{c,f}^{n}$
are disjoint open subsets of $X$ that separate $(c_0,d)$ and $(c,f)$. 
Finally, since $C\times D$ is Hausdorff and open in $X,$ any two points in $C\times D$ can be separated by disjoint open subsets of $X$.
This finishes the proof that $X$ is Hausdorff. 
\end{proof}

\begin{claim}
\label{claim:zero-dimensional}
$X$ is locally compact and zero-dimensional.
\end{claim}
\begin{proof}
Let $(c,d)\in C\times D$ be arbitrary. Consider an open subset $U$ of $X$ containing $(c,d)$. Then $U\cap (C\times D)$ is an open subset of $C\times D$ containing $(c,d)$. Since $C\times D$ is locally compact and zero-dimensional,  there exists a clopen compact subset $K$ of $C\times D$ 
such that $(c,d)\in K\subseteq U\cap (C\times D)$.
Since $K$ is open in $C\times D$ and $C\times D$ is open in $X$, $K$ is open in $X$. Since $K$ is compact and $X$ is Hausdorff by Claim~\ref{Hausdorff:claim}, $K$ is closed in $X$. Thus, $K$ is a compact clopen subset of $X$ such that $(c,d)\in K\subseteq U$.

Let $(c,f)\in C\times \mathscr{F}$ 
be arbitrary. Since $\mathcal{B}_{c,f}$ is a local base of $X$ at
$(c,f)$, in view of  \eqref{eq:B}, it suffices to check that
each $B_{c,f}^k$ is a compact clopen subset of $X$.

Fix  $k\in\N$. By definition, $B_{c,f}^k$ is an open subset of $X$. 
Let us check that $B_{c,f}^k$is a closed subset of $X$.
Since $B_{c,f}^k\cap (C\times D)=O_{c,f}^k$ by \eqref{eq:B:element},
the latter set is closed in $C\times D$  and $C\times D$ is open in $X$, it follows
that every point $(c,d)\in (C\times D)\setminus B_{c,f}^k$ 
has an open neighborhood in $X$ disjoint from $B_{c,f}^k$.
Let $(c',f')\in (C\times \mathscr{F})
\setminus 
B_{c,f}^k$.
Then $(c,f)\not=(c',f')$, so we can apply Claim \ref{disjoint:claim}
to find $n\in\N$ such that
$O_{c,f}^n\cap O_{c',f'}^{n}=\emptyset$.
If $n\le k$, then $O_{c,f}^k\subseteq O_{c,f}^n$ and 
$O_{c',f'}^{k}\subseteq O_{c',f'}^{n}$ by \eqref{eq:0}, 
so the sets $O_{c,f}^k$ and $O_{c',f'}^{k}$ are disjoint, which implies
$B_{c,f}^k\cap B_{c',f'}^{k}=\emptyset$ by \eqref{eq:B:element}.
Assume now that $k< n$.
Then 
$$
B_{c,f}^k\setminus B_{c,f}^n=
O_{c,f}^k\setminus O_{c,f}^n
=
\bigcup_{m\in A, k< m\le n} M_{c,f}^m
=
\bigcup_{m\in A, k< m\le n} V_m^c\times \{f(m)\}
$$
by \eqref{eq:M}, \eqref{eq:0} and \eqref{eq:B:element},
so this set is compact, and thus closed in $X$.
Now $B_{c',f'}^{n}\setminus (B_{c,f}^k\setminus B_{c,f}^n)$ is an open neighbourhood of $(c',f')$
disjoint from $B_{c,f}^k$.

We have proved that $B_{c,f}^k$ is a clopen subset of $X$.
Since $B_{c,f}^k\setminus B_{c,f}^n$ is compact whenever $k<n$, it follows that each $B_{c,f}^k$ is compact.
\end{proof}

Since $X$ is Hausdorff  (Claim \ref{Hausdorff:claim}) and zero-dimensional (Claim \ref{claim:zero-dimensional}), it is Tychonoff.

\begin{claim}
\label{claim:closed:and:discrete}
For every $c^*\in C$, the set $Z_{c^*}=\{c^*\}\times D$ is discrete and closed in $X$.  
\end{claim}
\begin{proof}
Clearly, $Z_{c^*}$ is discrete in $C\times D$, and thus also in $X$.
Furthermore, $Z_{c^*}$ is obviously closed in $C\times D$. So it remains only to show that no point $(c,f)\in X\setminus (C\times D)=C\times \mathscr{F}$
lies in the closure of $Z_{c^*}$.
Fix a point $(c,f)\in C\times \mathscr{F}$.
By Claim \ref{claim:disjoint:partition},
there exists at most one $n\in \N $ such that $c^{*}\in V_{n}^c.$ (If no such $n$ exists, we define $n=1$.)
By 
\eqref{eq:M}, \eqref{eq:0} and \eqref{eq:B:element},
$B_{c,f}^{n}$ 
does not intersect $Z_{c^{*}}.$ 
Since $B_{c,f}^{n}\in \mathcal{B}_{c,f}$ by \eqref{eq:B}, it is an open neighbourhood of $(c,f)$ in $X$.
\end{proof}

\begin{claim}
Player $\B$ does not have a stationary winning strategy in $\Sp(X)$.
\end{claim}
\begin{proof}
By Theorem \ref{stationary_winning_strategy}, 
it suffices to show that $X$ does not have a dense relatively countably compact subset. Let $Y$ be a dense subset of $X$.
For every 
$d\in D$, 
the set $U_d=C\times \{d\}$ is open in $X$, and since $Y$ is dense in $X$,
there exists 
$c_d\in C$ such that
$(c_d,d)\in U_d
\cap Y$.
 Since $|C|=\mathfrak{c},$ 
there exist $c^*\in C$ and a faithfully indexed set $\{d_n:n\in \N\}$ such that
$c^*=c_{d_n}$ for every $n\in\N.$  
Clearly, $S=\{(c^*,d_n):n\in\N\}\subseteq Z_{c^*}$. 
Since
$Z_{c^*}$ is a closed discrete subspace of $X$ by Claim \ref{claim:closed:and:discrete},  $S$ has no accumulation points in $X$.
Since $S$ is contained in $Y$, we conclude that $Y$ is not relatively countably compact in $X$.
\end{proof}

\begin{claim}
\label{subsequence:claim}
Suppose that $J$ is an infinite subset of $\N$ and $\{c_j:j\in J\}$ is a sequence in $C$ converging to $c\in C$ such that
$c_l\not= c_m$ whenever $l\not=m$. Then there exist strictly increasing functions 
$j:\N\to J$
and
$k:\N\to\N$
such that
$c_{j(m)}\in V_{k(m)}^c$ for every $m\in \N$.
\end{claim}
\begin{proof}
Without loss of generality, we shall assume that $c_j\not=c$ for every $j\in\N$.
Since $\{V_n^c:n\in\N\}$ is a partition of $C\setminus\{c\}$
by Claim \ref{claim:disjoint:partition},
each $c_j$ is contained in exactly one element $V_{n_j}^c$ of this partition. Moreover,
since $V_n^c$ is a clopen subset of $C$ and the sequence $\{c_j:j\in J\}$ converges to $c\not\in V_n^c$, each $V_n^c$ contains at most finitely many elements of the sequence $\{c_j:j\in J\}$.

By induction on $m\in \N$,
we shall define $j(m)\in J$ and $k(m)\in \N$ such that:
\begin{itemize}
\item[(1$_m$)]
$c_{j(m)}\in V_{k(m)}^c$;
\item[(2$_m$)]
if $m\ge 2$, then
$j(m)>j(m-1)$ and $k(m)>k(m-1)$.
\end{itemize}

Let $j(1)\in J$ be arbitrary. Define $k(1)=n_{j(1)}$.
Then (1$_1$) and (2$_1$) hold.

Let $m\geq 2$ and suppose that 
$j(s)\in J$ and $k(s)\in \N$ satisfying (1$_s$) and (2$_s$)
have already been defined 
for every $s\leq m-1$. 
The set $\bigcup_{n\le k(m-1)} V_n^c$ contains only finitely many elements of the sequence $\{c_j:j\in J\}$. Since $J$ is infinite,
we can find $j(m)\in J$ such that $j(m-1)<j(m)$ and
$c_{j(m)}\not\in \bigcup_{n\le k({m-1})} V_n^c$.
Let $k(m)=n_{j(m)}$.
Then $c_{j(m)}\in V_{n_{j(m)}}^c=V_{k(m)}^c$,
which implies $k(m-1)<k(m)$.
Thus, (1$_m$) and (2$_m$) hold.
\end{proof}

\begin{claim}
\label{claim:1}
If $x_n=(c_n,d_n)\in C\times D=M$ for every $n\in\N$ and 
$c_n\not=c_m$ whenever $m,n\in\N$ and $m\not=n$, then the sequence
$\{x_n:n\in\N\}$ has a subsequence converging in $X$.
\end{claim}
\begin{proof}
{\em Case 1\/}. {\sl There exists $d\in D$ such that $N_d=\{n\in\N: d_n=d\}$ is infinite.\/}
Since $\{c_n:n\in N_d\}$ is a sequence of elements of the Cantor set $C$, there exists an infinite set $K\subseteq N_d$ such that the sequence  $\{c_n:n\in K\}$ converges to some $c\in C$.  
Now the subsequence $\{x_n:n\in K\}$ of the sequence $\{x_n:n\in\N\}$ converges to the point $(c,d)\in C\times D$.

{\em Case 2\/}. {\sl The set $N_d=\{n\in\N: d_n=d\}$ is finite for each $d\in D$.\/}
In this case, we can choose an infinite set $N\subseteq \N$ such that
$d_m\not=d_n$ whenever $m,n\in N$ and $d_m\not=d_n$. 
Since $C$ is compact metric, there is an infinite subset
$J\subseteq N$ such that the sequence 
$\{c_n:n\in J\}$ converges to some point $c\in C.$  
Let 
$j$ and $k$ 
be as in the conclusion of Claim 
\ref{subsequence:claim}. Since $k$ is an injection, the set  
$P=k(\N)$ is infinite as well.
Define
$p:P\to D$ by 
\begin{equation}
\label{eq:def:p}
p(m)=d_{j\circ k^{-1}(m)}
\ \text{ for }\ 
m\in P.
\end{equation}
Since $j$ and $k$ are injective,  $p$ is well defined  and injective.  
By Claim \ref{van:douwen}, there is $f\in \mathscr{F}$ such that 
the set $T=\{m\in P:p(m)=f(m)\} $ is infinite. Therefore the set $S=k^{-1}(T)$ is infinite. 
Let $n\in S.$ Then  $f(k(n))=p(k(n))=d_{j(n)}.$ 
Since $c_{j(n)}\in V_{k(n)}^c,$ 
then
\begin{equation}
\label{eq:xjm}
x_{j(n)}=(c_{j(n)},d_{j(n)})\in V_{k(n)}^c\times\{f(k(n))\}=
M_{c,f}^{k(n)}
\end{equation}

It remains only to observe that
the sequence 
$\{x_{j(n)}:n\in S\}$ converges to the point
$(c,f)$ in $X.$ 
Indeed, let $m\in\N$ be arbitrary.
Since the function $k$ is monotonically increasing, 
there exists $l\in\N$ such that $k(n)>m$ provided that $n\ge l$.
It follows from 
\eqref{eq:0}, \eqref{eq:B:element} and
\eqref{eq:xjm}
that $x_{j(n)}\in M_{c,f}^{k(n)}\subseteq O_{c,f}^m\subseteq B_{c,f}^m$ for $n\in S$ and
$n\ge l$.
\end{proof}

\begin{claim}
Player $\B$ has a winning strategy in $\Ssp(X).$ 
\end{claim}
\proof
We define a strategy $\TAU: \Seq (\mathcal{O}(X))\setminus\{\emptyset\}\to  X$ for Player $\B$
by induction on the length of the sequence $(V_1,\dots,V_n)\in \Seq (\mathcal{O}(X))\setminus\{\emptyset\}$.

For every 
$V\in \mathcal{O}(X)$ select a point $\TAU(V)\in V\cap (C\times D).$  
This can be done as $C\times D$ is dense in $X$.

Let $m\in \N$ and suppose that for every $n\in\{1,\ldots,m\}$ and 
$(V_1,\ldots,V_n)\in \Seq (\mathcal{O}(X))\setminus\{\emptyset\}$
we have already selected a point $\TAU(V_1,\ldots,V_n)\in V_n\cap (C\times D).$
Let $(V_1,\ldots,V_{m+1})\in\Seq (\mathcal{O}(X))\setminus\{\emptyset\}.$
For every $n\in\{1,\ldots,m\}$ define $y_n=\TAU(V_1,\ldots,V_n).$ 
Define $F=\pi(\{y_1,\dots,y_{m}\}),$ where $\pi:C\times D\to C$ is the projection in the first coordinate. 
Then set $F \times D$ is closed and nowhere dense in $C\times D$, so 
$(V_{m+1}\cap (C\times D))\setminus(F\times D)\not=\emptyset.$
Therefore we can select a point 
\begin{equation}
\label{eq:strategy}
\TAU(V_1,V_2,\ldots,V_{m+1})\in (V_{m+1}\cap (C\times D))\setminus (F\times D).
\end{equation}

To show that the strategy $\TAU$ is winning for $\B$, 
let $\SIGMA:\Seq(X)\to \mathcal{O}(X)$ be an arbitrary strategy for Player $\A$ in $\Ssp(X).$
Let $w_{\SIGMA,\TAU}=(V_1,y_1,V_2,y_2,\ldots)$ be the play produced by $\SIGMA$ and $\TAU$. 

For every $n\in\N$, let $y_n=(c_n,d_n)\in C\times D$. It follows from
\eqref{eq:strategy} that the sequence $\{y_n:n\in\N\}$ satisfies the assumptions of Claim  \ref{claim:1}, applying which 
we conclude that the sequence $\{y_n:n\in\N\}$ has a convergent subsequence in $X$.
This proves that Player $\B$
wins the play $w_{\SIGMA,\TAU}$. 
Since $\SIGMA$ was arbitrary, $\TAU$ is a winning strategy for Player $\B$. 
The proof is complete.
\end{proof}

The next corollary 
shows that arrow $(c_1)$ of Diagram 3
is not reversible. 

\begin{cor}
\label{cor:5.11}
Let $\mathcal{S}$ be a topological property of sequences which is weaker than the property from item (i) of Example \ref{examples}
and stronger than the property from item (v) of the same example. 
Let $X$ be the space from Theorem \ref{arrow:3}.
Then 
Player $\B$ has a winning strategy in the game $\OP(X,\mathcal{S})$ but does not have a stationary winning strategy in the same game.
\end{cor}
\begin{proof}
Let $\mathcal{S}_1$ denote the property from item (i) of Example \ref{examples}.
By Definition \ref{def:Ssp:Sp}(i),
the game $\Ssp(X)$ is precisely the game $\OP(X,\mathcal{S}_1)$. 
Since $\B$ has a winning strategy in the game $\Ssp(X)$ 
by Theorem \ref{arrow:3},
this means
that $\B$ has a winning strategy in the game $\OP(X,\mathcal{S}_1)$.
Since $\mathcal{S}$ is weaker than  $\mathcal{S}_1$,
Proposition \ref{winning:strategies:for:R:and:S}(ii)
implies that $\B$ has a winning strategy in the game $\OP(X,\mathcal{S})$.

Assume that $\B$ has a stationary winning strategy $\TAU$ 
in $\OP(X,\mathcal{S})$.
Since $\mathcal{S}$ is 
stronger than the property $\mathcal{S}_2$ from item (v) of Example \ref{examples}, 
Proposition \ref{winning:strategies:for:R:and:S}(ii)
implies that
$\TAU$ is also a winning strategy for $\B$ in the game $\OP(X,\mathcal{S}_2)$.
By Definition \ref{def:Ssp:Sp}(ii),
the game $\OP(X,\mathcal{S}_2)$ coincides with $\Sp(X)$.
We conclude that $\B$ has a stationary winning strategy in 
the game $\Sp(X)$ on $X$, in contradiction with the conclusion of 
Theorem \ref{arrow:3}.
\end{proof}

\section{Example showing that arrow $(c_4)$ of Diagram 3 is not reversible}

\label{exa:c}
\label{sec:6}

\begin{theorem}\label{berner}
There exists a \ssp\ space $X$ such that  Player $\A$ has a 
winning strategy in $\Sp(X)$.
\end{theorem}
\begin{proof}
A. J. Berner constructed a pseudocompact space $X$ that does not contain a dense relatively countably compact subspace; 
see \cite[Section 3]{B}. Let us describe this space.
For every $\alpha\in\omega_1$, define 
\begin{equation}
\label{eq:X_alpha}
X_\alpha=\{x\in 2^{\omega_1}: x(\alpha)=1
\text{ and }
x(\gamma)=0
\text{ for all }
\gamma\in\omega_1
\text{ with }
\gamma>\alpha\}.
\end{equation}
Note that $X_\alpha$ is homeomorphic to the Cantor set $2^\omega$ for every $\alpha\ge\omega$.

We are going to show that the subspace
\begin{equation}
\label{eq:union}
X=\bigcup_{\alpha<\omega_1} X_\alpha
\end{equation}
of $2^{\omega_1}$ has the desired properties.

Note that $X_\alpha\cap X_\beta=\emptyset$ whenever $\alpha,\beta\in\omega_1$ and $\alpha\not=\beta$. This implies the following
\begin{claim}
\label{claim:1:z}
For every 
$x\in X$, there exists a unique $\alpha(x)\in\omega_1$ such that $x\in X_{\alpha(x)}$.
\end{claim}

\begin{claim}
\label{claim:2:z}
If $x\in X$, $\beta<\omega_1$ and $x(\beta)=1$, then $\beta\le \alpha(x)$.
\end{claim}
\begin{proof}
Since
$x\in X_{\alpha(x)}$
by Claim \ref{claim:1:z},
we conclude from \eqref{eq:X_alpha} that $x(\gamma)=0$
for all 
$\gamma\in\omega_1$
 with 
$\gamma>\alpha(x)$.
Since $x(\beta)=1$ by our assumption, the inequality $\beta\le\alpha(x)$ must hold.
\end{proof}

\begin{claim}
\label{hit:by:the:same:X:alpha}
For every sequence $\{W_n:n\in\N\}$ of non-empty open subsets of $2^{\omega_1}$, there exists $\alpha\in\omega_1\setminus\omega$ such that
$W_n\cap X_\alpha\not=\emptyset$ for every $n\in\N$.
\end{claim}
\begin{proof}
Without loss of generality, we may assume that each $W_n$ is a basic subset of $2^{\omega_1}$, so it has finite support $\mathrm{supp}(W_n)$. Therefore, the set
$C=\bigcup\{$supp$(W_n):n\in \N\}$ is at most countable, so 
$\alpha=\sup C+\omega+1\in\omega_1$.
An easy check that this $\alpha$ works is left to the reader.
\end{proof}

\begin{claim}
$X$ is \ssp.
\end{claim}
\begin{proof}
Let $\{U_n:n\in\N\}$ be a sequence of non-empty open subsets of $X$. For each $n\in\N$, fix an open subset $W_n$ of $2^{\omega_1}$ such that $U_n=X\cap W_n$. Clearly, $W_n$ is non-empty.
Let $\alpha\in\omega_1$ be the ordinal as in the conclusion of Claim \ref{hit:by:the:same:X:alpha}.
Then $X_\alpha$ is compact metric (being homeomorphic to the Cantor set $2^\omega$), so it is \ssp. Moreover, $X_\alpha\subseteq X$ by 
\eqref{eq:union}, so
$U_n\cap X\supseteq
U_n\cap X_\alpha=W_n\cap X\cap X_\alpha=W_n\cap X_\alpha\not=\emptyset$ for every $n\in\N$. Now the conclusion of our claim follows from 
\cite[Lemma 3.3]{DoS}.
\end{proof}

For every $\alpha<\omega_1$, 
\begin{equation}
\label{eq:V:alpha}
V_\alpha=\{x\in X: x(\alpha)=1\}
\end{equation}
 is an
open subset of $X$.

Define the strategy $\SIGMA:\Seq(X)\to \mathcal{O}(X)$ 
for $\A$
by $\SIGMA(\emptyset)=X$ and 
\begin{equation}\label{eq:W}
\SIGMA(x_1,x_2,\dots,x_n)=V_{\alpha(x_n)+1}\cap\bigcap_{i=1}^{n} V_{\alpha(x_i)}
\text{ for }
(x_1,x_2,\dots,x_n)\in \Seq(X)\setminus\{\emptyset\}.
\end{equation}
(The fact that this set is non-empty follows from
Claim \ref{hit:by:the:same:X:alpha}
and \eqref{eq:union}.)

To prove 
that $\SIGMA$ is a winning strategy for $\A$ in $\Sp(X)$, let $\TAU$ be an arbitrary strategy for  $\B$ in $\Sp(X)$.  
Let $w_{\SIGMA,\TAU}=(U_1,x_1,U_2,x_2,\dots,U_n,x_n,\dots)$
be the play produced by following strategies $\SIGMA$ and $\TAU$.
Recall that 
$U_1=\SIGMA(\emptyset)=X$
and 
$x_1=\TAU(U_1)$ 
by \eqref{eq:U_1} and \eqref{eq:x_1},
and $U_n$ and $x_n$ for $n\in\N^+$ are defined
by \eqref{eq:U_n} and
\eqref{eq:7z}.

\begin{claim}\label{claim:berner}
If $j,n\in\N$
and $1\le j\le n$, then
$x_{n+1}(\alpha(x_j))=1$ and $\alpha(x_n)<\alpha(x_{n+1}).$
\end{claim}
\begin{proof}
Note that $x_{n+1}=\TAU(U_1,\dots,U_{n+1})\in U_{n+1}=\SIGMA(x_1,\dots,x_{n})$
by \eqref{eq:7z}, \eqref{eq:rule:of:the:game}
and
\eqref{eq:U_n}.
Combining this with 
\eqref{eq:W},
we get 
$x_{n+1}\in V_{\alpha(x_n)+1}\cap V_{\alpha(x_j)}$.
From this and \eqref{eq:V:alpha},
we conclude that
$x_{n+1}(\alpha(x_n)+1)=x_{n+1}(\alpha(x_j))=1$.
Since $x_{n+1}(\alpha(x_n)+1)=1$, 
applying Claim \ref{claim:2:z} with $x=x_{n+1}$ and $\beta=\alpha(x_n)$, 
we conclude that
$\alpha(x_n)<\alpha(x_n)+1\le \alpha(x_{n+1})$.
\end{proof}

\begin{claim}
\label{claim:no:accumulation:points}
The sequence $\{x_n:n\in\N\}$ does not have an accumulation point in $X$.
\end{claim}
\begin{proof}
Suppose 
that $x\in X$ is an accumulation point of the sequence $\{x_n:n\in\N\}$.

Let $j\in \N$ be arbitrary.
Since $x$ is an accumulation point of $\{x_n:n\in\N\}$,
the $\alpha(x_j)$th coordinate 
$x(\alpha(x_j))$ of $x$ is an accumulation point of the sequence
$\{x_{n+1}(\alpha(x_j)):n\in\N\}$, and therefore, also an accumulation point of its cofinal subsequence
$\{x_{n+1}(\alpha(x_j)):n\geq j\}$.
Since $x_{n+1}(\alpha(x_j))=1$ for all $n\ge j$ by 
Claim \ref{claim:berner},
it follows that $x(\alpha(x_j))=1$.
Therefore, $\alpha(x_j)\leq \alpha(x)$ by Claim \ref{claim:2:z}.

Let $n\in\N$ be arbitrary.
By Claim \ref{claim:berner}, 
$\alpha(x_n)<\alpha(x_{n+1})$.
As was shown in the preceding paragraph,
$\alpha(x_{n+1})\leq \alpha(x)$.
Since $x_n\in X_{\alpha(x_n)}$ by Claim \ref{claim:1:z}
and $\alpha(x)>\alpha(x_n)$, from 
\eqref{eq:X_alpha} we conclude that
$x_n(\alpha(x))=0$.

Since $x$ is an accumulation point of $\{x_n:n\in\N\}$,
the $\alpha(x)$th coordinate 
$x(\alpha(x))$ of $x$ is an accumulation point of the sequence
$\{x_{n}(\alpha(x)):n\in\N\}$.
Since $x_{n}(\alpha(x))=0$ for every $n\in\N$, we conclude that
$x(\alpha(x))=0$.
On the other hand,
$x\in X_{\alpha(x)}$ by Claim \ref{claim:1:z}, so
$x(\alpha(x))=1$
by \eqref{eq:X_alpha}.
This contradiction finishes the proof of our claim.
\end{proof}

By Claim \ref{claim:no:accumulation:points},
the sequence $\{x_n:n\in\N\}$ does not have an accumulation point in $X.$ This proves that $\A$
wins the  play $w_{\SIGMA,\TAU}$. Since $\TAU$ was an arbitrary strategy for $\B$ in $\Sp(X)$, this proves that $\SIGMA$ is a winning strategy for  $\A$ in $\Sp(X)$.
\end{proof}

The next corollary shows 
that arrow 
$(c_4)$ of Diagram 3 is not reversible.
This result is due to Y.~Hirata \cite{Hirata}; see Remark \ref{historic:remark} below. 

\begin{cor}\label{not:arrow:c4}
Let $\mathcal{S}$ be a topological property of sequences preserved by 
projections which is 
weaker 
than the 
property 
from item (i) and stronger than the property from item (v) of Example \ref{examples}. 
Then there exists a selectively $\mathcal{S}$ space $Z$ such that
Player $\A$ has a stationary winning strategy in the game $\OP(Z, \mathcal{S})$.
\end{cor}

\begin{proof}
Let $X$ be the space from Theorem \ref{berner}. Then $X$ is \ssp. Observe that $\alpha(X_{\disc})$ is \seqc.
By \cite[Lemma 4.1]{DoS}, $Z=\alpha(X_{\disc})\times X$ is \ssp. 
Let $\mathcal{S}_1$ be the property from item (i) of Example \ref{examples}.
Since $Z$ is \ssp, $Z$ is selectively $\mathcal{S}_1$ by Remark \ref{names:remark}(i).
Since $\mathcal{S}$ is weaker than  $\mathcal{S}_1,$
$Z$ is selectively $\mathcal{S}$. 

Let $\mathcal{S}_2$ be the property from item (v) of Example \ref{examples}. 
By Definition \ref{def:Ssp:Sp}(ii),
the game $\OP(X,\mathcal{S}_2)$ coincides with the game $\Sp(X)$.
By Theorem \ref{berner}, $\A$ has a winning strategy $\SIGMA$ in $\Sp(X)$, or equivalently, $\OP(X,\mathcal{S}_2)$.
Since $\mathcal{S}$ is stronger than $\mathcal{S}_2$,
Proposition \ref{winning:strategies:for:R:and:S}(i) implies that
$\SIGMA$ is a winning strategy for $\A$ in the game $\OP(X,\mathcal{S})$. By Theorem \ref{from:winning:to:stationary},
Player $\A$ has a stationary winning 
strategy in $\OP(Z,\mathcal{S}).$
\end{proof}

\begin{remark}
\label{historic:remark}
The second listed author presented Theorem \ref{berner}
at Yokohama Topology Seminar on October 27, 2017. Clearly, 
this theorem means that either arrow $(c_3)$ or arrow $(c_4$) of Diagram~3 is not reversible. At that time, the authors were not able to determine which of these two arrows is not reversible.
Soon thereafter, Y. Hirata proved that
the product $X\times \alpha(\omega_1)$ of the space $X$ from  
Theorem \ref{berner} and the one-point compactification $\alpha(\omega_1)$ of the discrete space of size $\omega_1$ 
is selectively sequentially pseudocompact, yet Player $\A$ has a stationary winning strategy in the game $\Sp(X\times  \alpha(\omega_1))$ on the product 
\cite{Hirata}, thereby establishing that arrow $(c_4$) of Diagram~3 is not reversible. 
The authors were inspired by an idea of Y.~Hirata \cite{Hirata} of using a product of a given space $X$ with the one-point compactification of a discrete space and subsequently proved Theorem~\ref {from:winning:to:stationary}, which in turn implies Corollary \ref{not:arrow:c4}.
\end{remark}

\section{Open questions}

We do not know if  arrows $(c_2)$ and $(c_3)$ of Diagram 3 are reversible.

\begin{question}
\label{que:7.1}
Let $\mathcal{S}$ be a topological property of sequences  
weaker 
than the 
property 
from item (i) and stronger than the property from item (v) of Example \ref{examples}. 
\begin{itemize}
\item[(i)]
Is the game $\OP(X,\mathcal{S})$ determined?
Equivalently, is arrow ($c_2$) in Diagram~3 reversible?
What can be said for properties $\mathcal{S}$ from items (i) and (v) of Example \ref{examples} themselves?
\item[(ii)]
Is arrow ($c_3$) in Diagram~3 reversible? What can be said for properties $\mathcal{S}$ from items (i) and (v) of Example \ref{examples} themselves?
\end{itemize}
\end{question}

\begin{question}
\label{question:new}
Let $\mathcal{S}$ be a topological property of sequences  
weaker 
than the 
property 
from item (i) and stronger than the property from item (v) of Example \ref{examples}. 
Do any of the four arrows from Diagram~3 become reversible for the game $\OP(X,\mathcal{S})$ on:
\begin{itemize}
\item[(i)]
 a compact space $X$?
\item[(ii)]
a topological group $X$?
\item[(iii)]
the function space $X$ of the form $C_p(Y,G)$ for a topological space $Y$ and a topological group $G$ such that $C_p(Y,G)$ is dense in $G^Y$? (Here $C_p(Y,G)$ denotes the topological group of all continuous functions from $Y$ to $G$ endowed with the topology of pointwise convergence \cite{SS}.)
\end{itemize}
What can be said for properties $\mathcal{S}$ from items (i) and (v) of Example \ref{examples} themselves?
\end{question}

Corollary \ref{dyadic:result} justifies the following question:
\begin{question}
Let $G$ be a group such that the closure of every countable subgroup of $G$ is compact.
Does Player $\B$ have a stationary winning strategy in the \ssp\ game $\Ssp(G)$ on $G$?
\end{question}

We finish with a stronger  version of \cite[Question 1.5]{DoS2}: 

\begin{question}
If  an Abelian group $G$ admits a pseudocompact group topology, does it then admit a group topology $\mathcal{T}$ such that 
Player $\B$ has a stationary winning strategy in:
\begin{itemize}
\item[(i)]  the \ssp\ game $\Ssp(G,\mathcal{T})$?
\item[(ii)] the selectively pseudocompact game $\Sp(G,\mathcal{T})$)?
\end{itemize}
\end{question}

Additional questions related to this topic can be found in \cite{DoS3}.

\medskip
\noindent
{\bf Acknowledgements:\/} 
We are grateful to Professor Franklin Tall for his kind suggestion to consider a game-theoretic version of our selective sequential pseudocompactness property from \cite{DoS} during his visit 
to Ehime University in December 2016. It is this suggestion which led us to the introduction of games $\Ssp(X)$ and $\Sp(X)$ from Section \ref{sec:5}. 

The second listed author would like to thank cordially Professor Yasushi Hirata for answering his question raised at the talk at Yokohama Topology Seminar on October 27, 2017 and sending the solution in \cite{Hirata}. An idea from \cite{Hirata} inspired the authors to obtain the general reduction result stated in Theorem~\ref{from:winning:to:stationary}. (We refer the reader to Remark \ref{historic:remark} for additional details.)

The second listed author would like to thank Professor Boaz Tsaban for his invitation to give an invited lecture presenting results of this paper at the Conference ``Frontiers of Selection Principles'' held during August 27 -- September 1, 2017 at Cardinal Stefan Wyszy\'{n}ski University (Warsaw, Poland).

This paper was written during the first listed author's stay at the Department of Mathematics 
of Faculty of Science of Ehime University (Matsuyama, Japan)
in the capacity of Visiting Foreign Researcher under the support by 
CONACyT, M\'exico: Estancia Posdoctoral al Extranjero 178425/277660. 
He would like to thank CONACyT for its support and the host institution for its hospitality.

\end{document}